\documentclass[11pt]{article}
\usepackage[utf8]{inputenc}
\usepackage{comment}
\usepackage{amssymb,amsmath,amsthm,graphicx,xcolor,mathtools,enumerate,mathrsfs}
\usepackage{enumitem,thmtools}
\usepackage{thm-restate}
\usepackage{cite}
\usepackage{booktabs}
\usepackage{comment}
\usepackage{float}
\usepackage{subcaption}
\usepackage[backref=section]{hyperref}
\usepackage[export]{adjustbox}
\definecolor{mycitegreen}{HTML}{00A99D}
\hypersetup{
  colorlinks=true,
  linkcolor={blue!85!black},   
  citecolor={red!85!black},   
  urlcolor=mycitegreen    
}
\usepackage[noabbrev]{cleveref}
 \usepackage[margin=2.54cm]{geometry}
\usepackage{amsthm}
\usepackage{placeins}
\usepackage{sectsty}
\usepackage{array}
\usepackage{multirow}
\sectionfont{\centering}
\crefrangelabelformat{lemma}{#3#1#4-#5#2#6}

\usepackage{tikz}
\usetikzlibrary{math,patterns}
\usetikzlibrary{snakes,calc,decorations.pathmorphing,through}
\tikzset{snake it/.style={decorate, decoration=snake}}

\theoremstyle{plain}
\newtheorem{theorem}{Theorem}[section]
\crefname{theorem}{Theorem}{Theorems}

\newtheorem{proposition}[theorem]{Proposition}
\crefname{proposition}{Proposition}{Propositions}

\newtheorem{corollary}[theorem]{Corollary}
\crefname{corollary}{Corollary}{Corollaries}

\newtheorem{lemma}[theorem]{Lemma}
\crefname{lemma}{Lemma}{Lemmas}

\crefname{conjecture}{Conjecture}{Conjectures}

\newtheorem{problem}[theorem]{Problem}
\crefname{problem}{Problem}{Problem}

\newtheorem{claim}[theorem]{Claim}
\crefname{claim}{Claim}{Claims}

\crefname{observation}{Observation}{Observations}

\crefname{setup}{Setup}{Setups}

\crefname{fact}{Fact}{Facts}

\crefname{remark}{Remark}{Remarks}

\crefname{example}{Example}{Examples}

\theoremstyle{definition}
\newtheorem{definition}[theorem]{Definition}
\crefname{definition}{Definition}{Definitions}

\crefname{construction}{Construction}{Constructions}

\crefname{question}{Question}{Questions}

\crefname{section}{Section}{Sections}
\Crefname{section}{Section}{Sections}

\crefname{subsection}{Subsection}{Subsections}
\Crefname{subsection}{Subsection}{Subsections}

\crefname{subsubsection}{Subsubsection}{Subsubsections}
\Crefname{subsubsection}{Subsubsection}{Subsubsections}

\crefname{appendix}{Appendix}{Appendices}
\Crefname{appendix}{Appendix}{Appendices}

\numberwithin{equation}{section}

\renewcommand{\int}[1]{\mathop{\mkern 0mu\mathrm{int}}\nolimits(#1)}

\newcommand{\keywords}[1]{\par\noindent\textbf{Keywords: }#1\par}
\newcommand{\subjclass}[1]{\par\noindent\textbf{Mathematics Subject Classifications: }#1\par}
\allowdisplaybreaks
\usepackage{listings}
\lstdefinestyle{Python}{
    language=Python,
}

\definecolor{DarkDesaturatedBlue}{HTML}{3A3556}
\definecolor{VividOrange}{HTML}{F15918}
\definecolor{PureOrange}{HTML}{FFBA00}
\definecolor{LightGrayishPink}{HTML}{EEC5D5}
\definecolor{VerySoftBlue}{HTML}{B5AFDB}

\usepackage{tikz}
\usetikzlibrary{math}
\usetikzlibrary{calc,decorations.pathmorphing,through}
\tikzset{snake it/.style={decorate, decoration=snake}}

\definecolor{DarkDesaturatedBlue}{HTML}{3A3556}
\definecolor{VividOrange}{HTML}{F15918}
\definecolor{PureOrange}{HTML}{FFBA00}
\definecolor{LightGrayishPink}{HTML}{EEC5D5}
\definecolor{VerySoftBlue}{HTML}{B5AFDB}
  \makeatletter
  \newcommand{\labelinthm}[1]{%
     \label{temp#1}
     \protected@write \@auxout {}{\string \newlabel{#1}{{\emph{\ref{temp#1}}}{\thepage}{\emph{\ref{temp#1}}}{temp#1}{}} }%
  }
  \makeatother

\title{Oriented Discrepancy of The Square of Hamilton Cycles}
\author{Yufei Chang\thanks{School of Mathematics, Shandong University, Jinan 250100, China. Supported by National Natural Science Foundation of China (No.12571373).}
\and Yangyang Cheng\thanks{Faculty of Computer Science and Mathematics, University of Passau, Passau, Germany. Partially supported by the Deutsche Forschungsgemeinschaft (DFG, German Research Foundation)-542321564.}
\and Zhilan Wang\footnotemark[1] \thanks{Corresponding author. Email: \href{mailto:zhilanwang@mail.sdu.edu.cn}{\texttt{zhilanwang@mail.sdu.edu.cn.}}}
\and Shuo Wei\footnotemark[1]
\and Jin Yan\footnotemark[1]}
\date{}

\begin{document}

\maketitle

\begin{abstract}
 For an oriented graph $G$, the oriented discrepancy problem concerns the existence of a spanning subgraph of $G$ with a large imbalance between its forward and backward edge orientations. Freschi and Lo proved the Dirac-type Hamilton cycle result in oriented graphs, and asked for an analogue for powers of Hamilton cycles under a minimum-degree condition. We show that, for sufficiently large $n$, every oriented graph $G$ on $n$ vertices with minimum degree $\delta(G) \ge 2n/3$ contains the square of a Hamilton cycle $H$ with $\sigma_{\max}(H)$ guaranteed to exceed a function depending on $\delta(G)$ and $n$. 
 

\vspace{1em}
\noindent

\vspace{1em}
\noindent
\end{abstract}

\section{Introduction}\label{sec:intro}

The study of spanning structures in dense graphs is a cornerstone of extremal graph theory. A classical theorem of Dirac \cite{Dirac} asserts that every graph $G$ on $n\ge 3$ vertices with minimum degree $\delta(G)$ at least $n/2$ contains a Hamilton cycle. A central object of interest in this domain is the $k$th power of a Hamilton cycle. For a graph $G$, the $k$th power of $G$ is the graph on the same vertex set in which two vertices are adjacent if and only if their distance in $G$ is at most $k$. The foundational threshold for the existence of $k$th power of Hamilton cycles is given by the celebrated P\'osa-Seymour conjecture, which was proved for large $n$ by Koml\'{o}s, S\'{a}rk\"{o}zy, and Szemer\'{e}di \cite{Komlos3}. 

\begin{theorem}[{[\citen{Komlos3}]}]\label{posa}
For every positive integer $k$, there exists $n_0\in \mathbb{N}$ such that every graph $G$ on $n\ge n_0$ vertices satisfying $\delta(G)\ge kn/(k+1)$ contains the $k$th power of a Hamilton cycle.
\end{theorem}

The corresponding problems in directed settings have also received considerable attention. For tournaments, Dragani\'c et al. \cite{Draganic1} obtained tight bounds, up to the relevant error term, for forcing powers of Hamilton cycles under minimum semi-degree conditions. More recently, DeBiasio et al. \cite{Debiasio}  studied powers of Hamilton cycles in oriented and digraphs. In particular, they showed that they showed that every digraph $D$ with minimum total degree $\delta(G)\geq(8/5+\varepsilon)n$ contains a square of Hamilton cycle and  a semi-degree condition $\delta^0(D)\ge (1/2-1/10^{6000k})n$ suffices to guarantee the $k$th power of a Hamilton cycle, provided that $D$ is sufficiently large. These results provide the natural
directed analogue of the P\'osa--Seymour problem. 

Beyond the mere existence of spanning structures, a rich line of inquiry asks whether one can find such subgraphs exhibiting a large bias in a two-edge-coloured graph. This kind of problem, known as discrepancy problems, traces back to Erd\H{o}s and Spencer \cite{Erdos1963,ErdosSpencer1972}. Discrepancy problems have since been resolved for a variety of structures, including spanning trees, graph factors, and Hamilton cycles \cite{Bal,BaloghC,Bradac2,Fre,Gishboliner,Gis,Gis1}. More recently, Brada\v{c} \cite{Bradac1} extended this to powers of Hamilton cycles, showing that for $k\ge 3$, a minimum degree of $(1-1/(k+1)+\eta)n$ guarantees the existence of the $k$th power of a Hamilton cycle with linear discrepancy.

An intriguing oriented analogue of graph discrepancy was introduced by Gishboliner, Krivelevich, and Michaeli \cite{Gis2}. In this setting, the role of the two colours is replaced by the two possible directions of each edge in a tournament or oriented graph. They conjectured that every oriented graph $G$ with minimum degree at least $n/2$ contains a Hamilton cycle with at least $\delta(G)$ edges in one direction. This conjecture was later confirmed by Freschi and Lo \cite{Freschi}. To formalize this, let $C = v_1 \ldots v_\ell v_1$ be a cycle in an oriented graph. We define the forward and backward edge counts as $\sigma^+(C):=\bigl|\{1\le i\le \ell:\; v_iv_{i+1}\in E(G)\}\bigr|$ and $\sigma^-(C):=\bigl|\{1\le i\le \ell:\; v_{i+1}v_i\in E(G)\}\bigr|$, where the indices are taken modulo $\ell$. We then measure the \emph{dominant} and \emph{minor direction} by setting $\sigma_{\max}(C):=\max\{\sigma^+(C),\sigma^-(C)\}$ and $\sigma_{\min}(C):=\min\{\sigma^+(C),\sigma^-(C)\}$. Analogous notation applies to paths and to the square of a Hamilton cycle, i.e. for the square of a Hamilton cycle $H$, define $\sigma^+(H):= \bigl|\{(i,j): j\in\{1,2\},\ v_iv_{i+j}\in E(G)\}\bigr|$ and $\sigma^-(H):= \bigl|\{(i,j): j\in\{1,2\},\ v_{i+j}v_i\in E(G)\}\bigr|$, where indices are taken modulo $n$. Set
$\sigma_{\max}(H):=\max\{\sigma^+(H),\sigma^-(H)\}$.

\begin{theorem}[{[\citen{Freschi}]}]\label{lo}
     Let $G$ be an oriented graph on $n \geq 3$ vertices. If $\delta(G) \geq \frac{n}{2}$ then there exists a Hamilton cycle $C$ in $G$ such that $\sigma_{\max}(C) \geq \delta(G)$.
 \end{theorem}

An Ore-type analogue of  \cref{lo} was proposed as a conjecture by Ai et al.~\cite{Ai}.
Recall that, for a graph $G$, the Ore-type condition is defined as
$\sigma_2(G) := \min\{d(x) + d(y) \mid x,y \in V(G),\ x \text{ and } y \text{ are non-adjacent}\}$.
Recently, relevant progress has been made on this conjecture; \cite{Ai, Chang1}.
In particular, Chang et al.~\cite{Chang1} solved the conjecture asymptotically when $n$ is sufficiently large.

\begin{theorem}[{[\citen{Chang1}]}] \label{th2}
For every $\gamma > 0$, there exists a positive integer $n_0$ such that the following holds.
Let $G$ be an oriented graph on $n \ge n_0$ vertices with $\sigma_2(G)\ge n$.
Then $G$ contains  a Hamilton cycle $C$ such that  \[\sigma_{\max}(C) \ge  \max\left\{\frac{n}{2},\frac{\sigma_2(G)}{2}-\gamma n\right\}.\]
\end{theorem} 

In view of the classical minimum degree condition in \cref{posa} and the oriented discrepancy result of Freschi and Lo in \cref{lo}, it is natural to ask for analogous discrepancy guarantees for the square of a Hamilton cycle in oriented graphs. This question was explicitly raised by Freschi and Lo \cite{Freschi}.

\begin{problem}[{[\citen{Freschi}]}]\label{pro2}
Let $G$ be an oriented graph on $n$ vertices with $\delta(G)\ge 2n/3$, and $H$ be the square of a Hamilton cycle in $G$. Determine a maximum lower bound on $\sigma_{\max}(H)$.
\end{problem}

Our main result resolves \cref{pro2} by establishing a general lower bound on $\sigma_{\max}(H)$ in terms of a function $f(\delta(G),n)$.

\begin{theorem}\label{th1}

For every $\xi>0$, there exists $n_0\in\mathbb{N}$ such that any oriented graph $G$ on $n\ge n_0$ vertices with $\delta(G)\ge 2n/3$ contains the square of a Hamilton cycle $H$ satisfying $\sigma_{\max}(H)\ge \max\{n,\; f(\delta(G),n)-\xi n\},$
where $f(\delta(G),n)>n$ is a function of $\delta(G)$ and $n$.

\end{theorem}

We further derive an explicit expression for $f(\delta(G),n)$ and show that this bound is larger than $n$. Moreover, under additional structural assumptions on $G$, the above estimate can be made more explicit. In particular, when $G$ is $T_6$-free (i.e., $G$ contains no tournament on $6$ vertices as a subgraph), our argument yields the following corollary.

\begin{corollary}\label{coro}
For every $\xi > 0$, there exists $n_0 \in \mathbb{N}$ such that every $T_6$-free oriented graph $G$ on $n \ge n_0$ vertices satisfying $\delta(G) \ge 2n/3$ contains the square of a Hamilton cycle $H$ with
\[
\sigma_{\max}(H)\ge \max\{n,\; 3\delta(G)-n-\xi n\}.
\]
\end{corollary}

\noindent\textbf{Organization.}
The remainder of this paper is organized as follows. In \cref{sec:prep}, we introduce the necessary notation and outline the proof method.
In \cref{sec:tool}, we collect the auxiliary  tools used throughout the paper. In \cref{sec:cac}, we establish the key almost-covering, absorbing, and connecting lemmas that underlie the proof of \cref{th1}. Ultimately, in \cref{proofmain}, we complete the proof of \cref{th1} and derive \cref{coro}.


\section{Preliminaries}\label{sec:prep}
\subsection{Notation}\label{subsec:notation}
For standard notation not defined in this paper, we refer the reader to \cite{Bang}. Let $G$ be a digraph. We write $V(G)$ for the vertex set of $G$, $E(G)$ for the edge set of $G$, and define $|G|:=|V(G)|$ and $e(G):=|E(G)|$. 
 For a vertex $v\in V(G)$, let $N_G^+(v)=\{u\in V(G):vu\in E(G)\}$ and $N_G^-(v)=\{u\in V(D):uv\in A(G)\}$ be the \emph{out-neighborhood} and \emph{in-neighborhood} of $v$, respectively. Their sizes are denoted by $d_G^+(v)=|N_G^+(v)|$ and $d_G^-(v)=|N_G^-(v)|$, for vertices in $G$, we often write $d^+(v)$ and $d^-(v)$ for short. 
The \textit{total degree} $d(v)$ in $G$ is the number of edges incident to $v$, i.e. $d(v):=d^+(v)+d^-(v)$, and the \textit{minimum total degree} of $G$ is denoted by $\delta(G)$. When there is no risk of confusion, we abbreviate minimum total degree to minimum degree.

An oriented graph is a digraph without loops such that between every two vertices there is at most one edge. A path or cycle is called \emph{directed} if all of its edges are directed consistently along the path or cycle. Unless stated otherwise, all paths and cycles in this paper are oriented (not necessarily directed). For brevity, we call the $k$th power of a path a $k$-path, and the $k$th power of a cycle a $k$-cycle. 

The underlying graph of $G$, denoted by $U(G)$, is the graph obtained from $G$ by ignoring the directions of its edges. 

Given two vertices $x,y\in V(G)$, we write $xy$ for the edge directed from $x$ to $y$. Given subsets $A,B\subseteq V(G)$, not necessarily disjoint, let $E_G(A,B)$, or simply $E(A,B)$ when $G$ is clear from the context, denote the set of all edges $xy\in E(G)$ such that $x\in A$ and $y\in B$. We write $e_G(A,B):=|E_G(A,B)|$, and omit the subscript $G$ whenever there is no danger of confusion. For a set $X\subseteq V(G)$, we write $G[X]$ for the subdigraph of $G$ induced by $X$, and $G\setminus X$ for the subdigraph induced by $V(G)\setminus X$. If $A,B\subseteq V(G)$ are disjoint, then $G[A,B]$ denotes the spanning bipartite subdigraph of $G$ with vertex set $A\cup B$ and edge set $E_G(A,B)\cup E_G(B,A)$.

Given a (di)graph $G$ and $t\in \mathbb{N}$, we let $G(t)$ denote the $t$-blow-up of $G$. More precisely, $V(G(t)):=\{v^j : v\in V(G)\text{ and } j\in [t]\}$ and $E(G(t)):=\{v^m w^\ell : vw\in E(G)\text{ and } m,\ell\in [t]\}$. 

Any orientation of the complete graph $K_n$ is called a tournament on $n$ vertices and is often denoted by $T_n$. Let $\mathcal{T}_n$ be the set of all such tournaments.

For graphs or digraphs $G$ and $H$, an $H$-tiling in $G$ is a collection of vertex-disjoint subgraphs of $G$, each isomorphic to $H$. An $H$-factor in $G$ is an $H$-tiling covering all vertices of $G$. For $k\in\mathbb{N}$, we write $[k]=\{1,2,\dots,k\}$. For a family of disjoint paths \( \mathcal{P} = \{P_1, P_2, \ldots, P_t\} \), we define $\sigma_{\max}(\mathcal{P}) = \sum_{i=1}^t \sigma_{\max}(P_i).$ The same definition applies to a family of disjoint 2-cycles.

Throughout the paper, we omit floor and ceiling signs whenever they are not essential. The constants in the hierarchies used to state our results are chosen from right to left. For example, when we write $0<a\ll b\ll c\le 1$, we mean that there exist non-decreasing functions $f:(0,1]\to(0,1]$ and $g:(0,1]\to(0,1]$ such that the desired conclusion holds whenever $0<a,b,c\le 1$, $b\le f(c)$, and $a\le g(b)$.

\subsection{Proof Overview}\label{sec:overview}

We now sketch the proof of \cref{th1}. The argument splits naturally based on whether $\delta(G)$ is close to the threshold $2n/3$. When $\delta(G)$ is close to $2n/3$, \cref{posa} yields the required discrepancy bound. Otherwise, we make use of the absorbing method introduced by R\"odl, Ruci\'nski and Szemer\'edi \cite{rodl}, which has become a fundamental tool for studying spanning structures in graphs, digraphs, and hypergraphs.

\medskip
\noindent\textbf{Step 1: Find an absorbing $2$-path.}
We first construct a short absorbing $2$-path $P_{\rm abs}$ with the property that it can absorb any sufficiently small leftover vertex set while preserving its end edges.

\medskip
\noindent\textbf{Step 2: Construct a reservoir set.}
Next we we choose a reservoir set $\mathcal R$, which will later be used to connect the $2$-paths produced in the covering step. The main part of the argument is then to analyze the remaining graph after removing $P_{\rm abs}$ and $\mathcal R$.

\medskip
\noindent\textbf{Step 3: Find an almost-spanning collection of $2$-paths.}
After deleting $V(P_{\rm abs})\cup \mathcal R$, we apply the Diregularity Lemma to find a mixed $\{T_r,T_{r-1}\}$-tilings in the reduced graph. We then use the Blow-up Lemma to lift these tiles back to the original graph and obtain a family $\mathcal Q$ of vertex-disjoint $2$-paths covering almost all remaining vertices.

\medskip
\noindent\textbf{Step 4: Connect the $2$-paths and absorb the remaining vertices.}
Finally, we use the reservoir $\mathcal R$ to connect the disjoint $2$-paths in $\mathcal Q$ and the absorbing path $P_{\rm abs}$ into a single spanning $2$-cycle. We then use $P_{\rm abs}$ to absorb all remaining uncovered vertices, thereby completing the square of a Hamilton cycle.
\medskip

 The key point is that the discrepancy in Step $3$ is preserved throughout the final assembly. Indeed, the family $\mathcal Q$ provides the dominant contribution to the discrepancy, whereas the connecting $2$-paths use only vertices from the reservoir $\mathcal R$ and therefore affect the discrepancy by at most a negligible error. Furthermore, these $2$-paths can be concatenated in an order that aligns with their dominant direction. Consequently, the resulting square of a Hamilton cycle attains the required discrepancy.

\section{Auxiliary tools}\label{sec:tool}

\subsection{The Diregularity Lemma}\label{subsec:regular}
 
We start with some definitions. The \emph{density} of bipartite graph \(G[A,B]\) with vertex classes \(A\) and \(B\) is defined to be
\[
d_G(A,B) \coloneqq \frac{e_G(A,B)}{|A||B|}.
\]
We often write \(d(A,B)\) if this is unambiguous. Given \(\varepsilon>0\), we say that \(G\) is \emph{\(\varepsilon\)-regular} if for all subsets \(X \subseteq A\) and \(Y \subseteq B\) with \(|X|>\varepsilon|A|\) and \(|Y|>\varepsilon|B|\) we have that \(|d(X,Y)-d(A,B)|<\varepsilon\). Given \(d \in [0,1]\) we say that \(G\) is \emph{\((\varepsilon,d)\)-super-regular} if it is \(\varepsilon\)-regular and furthermore \(d_G(a) \ge (d-\varepsilon)|B|\) for all \(a \in A\) and \(d_G(b) \ge (d-\varepsilon)|A|\) for all \(b \in B\). (This is a slight variation of the standard definition of \((\varepsilon,d)\)-super-regularity, where one requires \(d_G(a) \ge d|B|\) and \(d_G(b) \ge d|A|\).)

The Diregularity Lemma, proposed by Alon and Shapira \cite{Alon}, serves as the directed graph counterpart of the Regularity Lemma, with a proof strategy analogous to that of the undirected version. In this paper, we adopt the degree form of the Diregularity Lemma, which can be deduced from its standard formulation via the identical approach used to derive the degree form of the undirected Regularity Lemma (a sketch proof of this derivation is available in \cite{Kuhn}).

\begin{lemma}[Degree form of the Diregularity Lemma]\label{lem:diregularity}
For every \(\varepsilon \in (0,1)\) and $M^\prime\in\mathbb{N}$, there are integers $M=M(\varepsilon,M^\prime)\in\mathbb{N}$ and $n_0=n_0(\varepsilon,M^\prime)\in\mathbb{N}$ such that if \(G\) is a digraph on \(n \ge n_0\) vertices and \(d \in [0,1]\) is any real number, then there is a partition of the vertices of \(G\) into \(V_0, V_1, \ldots, V_k\) and a spanning subdigraph \(G^\prime\) of \(G\) such that the following holds:

\medskip

\textnormal{($i$)} \(M^\prime \le k \le M\),

\textnormal{($ii$)} \(|V_0| \le \varepsilon n\),

\textnormal{($iii$)} \(|V_1| = \cdots = |V_k| =: m\),

\textnormal{($iv$)} \(d_{G^\prime}^+(x) > d_G^+(x) - (d+\varepsilon)n\) for all vertices \(x \in G\),

\textnormal{($v$)} \(d_{G^\prime}^-(x) > d_G^-(x) - (d+\varepsilon)n\) for all vertices \(x \in G\),

\textnormal{($vi$)} for all \(1 \le i \le k\) the digraph \(G^\prime[V_i]\) is empty,

\textnormal{($vii$)} for all \(1 \le i, j \le k\) and \(i \ne j\) the bipartite graph \(G^\prime[V_i, V_j]\) whose vertex classes are \(V_i\) and \(V_j\) and whose edges are all the edges in \(G^\prime\) directed from \(V_i\) to \(V_j\) is \(\varepsilon\)-regular and has density either \(0\) or at least \(d\).
\end{lemma}

The vertex sets \(V_1, \ldots, V_k\) are called \emph{clusters}, \(V_0\) is called the \emph{exceptional set} and the vertices in \(V_0\) are called \emph{exceptional vertices}. The last condition of the lemma says that all pairs of clusters are \(\varepsilon\)-regular in both directions (but possibly with different densities). We call the spanning digraph \(G^\prime \subseteq G\) given by the Diregularity lemma the \emph{pure digraph}. Given clusters \(V_1, \ldots, V_k\) and a digraph \(G^\prime\), the \emph{reduced digraph} \(R^\prime\) with parameters \((\varepsilon,d)\) is the digraph whose vertex set is \([k]\) and in which \(ij\) is an edge if and only if the bipartite graph whose vertex classes are \(V_i\) and \(V_j\) and whose edges are all the \(V_i\)-\(V_j\) edges in \(G^\prime\) is \(\varepsilon\)-regular and has density at least \(d\). (Therefore if \(G^\prime\) is the pure digraph, then \(ij\) is an edge in \(R^\prime\) if and only if there is a \(V_i\)-\(V_j\) edge in \(G^\prime\)).

Note that the latter holds if and only if \(G^\prime[V_i, V_j]\) is \(\varepsilon\)-regular and has density at least \(d\). It turns out that \(R^\prime\) inherits many properties of \(G\), a fact that is crucial in our proof. However, \(R^\prime\) is not necessarily oriented even if the original digraph \(G\) is. 
That said, the following lemma demonstrate that we can construct a graph which almost inherits the minimum degree $\delta(G)$ of the original graph $G$.

\begin{lemma}[{[\citen{Kelly}]}]\label{dirprop}
Let \(G\) be an oriented graph of order at least \(n_0\) and let \(R^\prime\) be the reduced digraph by applying the Diregularity Lemma to \(G\) with parameters \(\varepsilon\), \(d\) and \(M^\prime\). Then \(R^\prime\) has a spanning oriented subgraph \(R\) with  \(\delta(R) \geq \left(\delta(G)/|G| - (3\varepsilon + 2d) \right) |R|\).
\end{lemma}

The oriented graph \(R\) given by Lemma \ref{dirprop} is called the \emph{reduced oriented graph}. The spanning oriented subgraph \(G^*\) of the pure digraph \(G^\prime\) obtained by deleting all the \(V_i\)-\(V_j\) edges whenever \(V_iV_j \in E(R^\prime) \setminus E(R)\) is called the \emph{pure oriented graph}. Given an oriented subgraph \(S \subseteq R\), the oriented subgraph of \(G^*\) corresponding to \(S\) is the oriented subgraph obtained from \(G^*\) by deleting all those vertices that lie in clusters not belonging to \(S\) as well as deleting all the \(V_i\)-\(V_j\) edges for all pairs \(V_i, V_j\) with \(V_iV_j \notin E(S)\).

The following Blow-up Lemma was proposed by Koml\'{o}s, S\'{a}rk\"{o}zy and Szemer\'{e}di \cite{Komlos1}. We will utilize this lemma to construct a path system with a specific structure that covers almost all vertices in $V(G)$.

\begin{lemma}[{[\citen{Komlos1}] Blow-up Lemma}] \label{blowup}
 For any graph $F$ with the vertex set $[k]$, and any positive numbers $d$ and $\Delta$, there exists $\sigma_0 = \sigma_0(d, \Delta, k)$ such that the following holds for all positive numbers $l_1,\ldots, l_k$ and all $0 < \sigma \leq \sigma_0$. Let $F^\prime$ be the graph obtained from $F$ by replacing each vertex $i\in V(F)$ with a set $V_i$ of $l_i$ new vertices and connecting all vertices between $V_i$ and $V_j$ whenever $ij\in E(F)$. If $G_1$ is a spanning subgraph of $F^\prime$ where every pair $(V_i, V_j )_{G_1}$ is $(\sigma, d)$-super-regular for every edge $ij\in E(F)$, then $G_1$ contains a copy of every subgraph $H$ of $F^\prime$ with $\Delta(H)\leq\Delta$.
 \end{lemma}

In order to apply  \cref{blowup}, it suffices to verify that the edges of a specific  subgraph in the reduced graph $R$ correspond to $(\varepsilon, d)$-super-regular cluster pairs. This is guaranteed by the following result from \cite{Kelly}.

\begin{proposition}[{[\citen{Kelly}]}]\label{prop:super-regular}
Let \(M^\prime, n_0, D\) be positive numbers and let \(\varepsilon, d\) be positive reals such that \(1/n_0 \ll 1/M^\prime \ll \varepsilon \ll d \ll 1/D\). Let \(G\) be an oriented graph of order at least \(n_0\). Let \(R\) be the reduced oriented graph with parameters \((\varepsilon,d)\) and let \(G^*\) be the pure oriented graph obtained by successively applying first the Diregularity Lemma with \(\varepsilon\), \(d\) and \(M^\prime\) to \(G\). Let \(U\) be an oriented subgraph of \(R\) with \(\Delta(U) \le D\). Let \(G^\prime\) be the underlying graph of \(G^*\). Then one can delete \(2D\varepsilon |V_i|\) vertices from each cluster \(V_i\) to obtain subclusters \(V_i^\prime \subseteq V_i\) in such a way that \(G^\prime\) contains a subgraph \(G^\prime_U\) whose vertex set is the union of \(V_i^\prime\) and such that

\textnormal{($i$)} \(G^\prime_U[V_i^\prime, V_j^\prime]\) is \((\sqrt{\varepsilon}, d-4D\varepsilon)\)-super-regular whenever \(ij \in E(U)\),

\textnormal{($ii$)} \(G^\prime_U[V_i^\prime, V_j^\prime]\) is \(\sqrt{\varepsilon}\)-regular and has density \(d-4D\varepsilon\) whenever \(ij \in E(R)\).

\end{proposition}

To embed a specified small graph across clusters with sufficient density and regularity, we rely on the following lemma.

\begin{lemma}[{[\citen{Komlos4}]}]\label{counting lemma}
For every \( d \) and \( h \) there exist \( \varepsilon = \varepsilon(d, h) \) and \( c = c(d, h) \), so that if all pairs of clusters among \( V_1, \ldots, V_h \) are of size at least \( c \), of density at least \( d \) and are \( \varepsilon \)-regular, then \( V_1, \ldots, V_h \) contain a copy of any graph \( H \) on \( h \) vertices (with one vertex in each cluster \( V_i \)).
\end{lemma}

 \subsection{Local Version of Hajnal-Szemer\'{e}di Theorem} \label{subsec:HSthm}
In this subsection, we record a recent refinement of the Hajnal--Szemer\'{e}di theorem due to Balogh et al.~\cite{Balogh}. This local version of the Hajnal--Szemer\'{e}di theorem allows us to partition the oriented graph into tournaments of size $r-1$ and $r$ with high precision, and serves as a key structural ingredient to our proof of \cref{th1}.

For $r \in \mathbb{N}$, we define the following parameters in an $n$-vertex oriented graph $G$:
\[
A_r := (r-1) \delta(G) - (r-2)n \quad \text{and} \quad \overline{A}_r := (r-1)n - r \delta(G).
\]
It is easy to check that if $(1 - 1/(r-1))n \le \delta(G) \le (1 - 1/r)n$, then $r A_r + (r-1) \overline{A}_r = n$.

\begin{lemma}[{[\citen{Balogh}]}] \label{HSnew}
Let $n, r \in \mathbb{N}$ and let $G$ be a graph on $n$ vertices. If
\[
\left(1 - \frac{1}{r-1}\right)n \le \delta(G) \le \left(1 - \frac{1}{r}\right)n,
\]
then $G$ contains a $K_r$-tiling consisting of $A_r$ copies of $K_r$ and a $K_{r-1}$-tiling consisting of $\overline{A}_r$ copies of $K_{r-1}$, such that the two tilings are vertex-disjoint.
\end{lemma}

By considering the underlying graph $U(G)$ of an oriented graph $G$, we immediately obtain the following directed analogue.

\begin{corollary}\label{HSnew1}
Let $n, r \in \mathbb{N}$ and let $G$ be an oriented graph on $n$ vertices. If
\[
\left(1 - \frac{1}{r-1}\right)n \le \delta(G) \le \left(1 - \frac{1}{r}\right)n,
\]
then $G$ contains a $T_r$-tiling of size $A_r$ and a $T_{r-1}$-tiling of size $\overline{A}_r$, which are vertex-disjoint.
\end{corollary}

\section{Almost covering, absorbing, and connecting Lemmas for Theorem~\ref{th1} } \label{sec:cac}
In this section, we  establish the key lemmas for the proof of \cref{th1}. To this end, we modify the original hypotheses by replacing degree conditions for graphs with the corresponding degree conditions for oriented graphs, while disregarding orientation in the absorption and connection steps. Under these modified assumptions, these arguments of~\cite{Levitt} remain valid. Throughout this section, we assume that \(1/n \ll \alpha \ll 1\); we do not repeat this hierarchy in the statements below.

\subsection{ Almost Covering Lemma.}
Before presenting our covering lemma, we first need some definitions and lemmas.

\begin{definition}\label{hk}

   For a positive integer $n$, a \textit{Hamilton square coupling} $S$ of a tournament $T_n$ is defined as follows. When $n \ge 5$, $S$ is the square of a Hamilton cycle. When $n = 3$ or $n = 4$, let $P = v_1^1 \dots v_n^1 v_1^2 v_2^2$ be a $2$-path in $T_n(2)$ and set $S = P - v_1^2 v_2^2.$ Moreover, let $\sigma_{\max}(T_i):=\sigma_{\max}(S)$, for $i=3$ or $4$.
   
\end{definition}

To further characterize the structural properties of tournaments, we introduce the following definition.

\begin{definition}\label{unav}
 For each $T_n\in\mathcal T_n$, let $\mathcal{S}_n$ be the family of all Hamilton square couplings of $T_n$. Define $$N_n:=\min_{T_n \in\mathcal T_n}\ \max_{S\in \mathcal{S}_n}\sigma_{\max}(S).$$
Equivalently, $N_n$ is the largest integer $k$ such that every tournament on $n$ vertices contains a Hamilton square coupling $S$ satisfying $\sigma_{\max}(S)\geq k$.
We also set $M_n = 2n - N_n$.
\end{definition}

We now establish the following properties of the parameters $N_n$ and $M_n$.

 \begin{lemma}[{[\citen{Draganic2}]}]\label{longpath}
 For \(n \ge 1\), every \(n\)-vertex tournament contains the square of a directed path of length $\ell=\lceil 2n/3\rceil-1$, but not necessarily of length \(\ell+1\).
 \end{lemma}

\begin{proposition}\label{Cr}
If $ n=3,4 $ or $5$,  $M_n= 3$ and $N_n=2n-3.$
\end{proposition}

\begin{proof}
 For the two non-isomorphic tournaments on 3 vertices, it is clear that $M_3 = 3$. For $n=4$, let $v_1v_2v_3v_4$ be a directed Hamilton path, and let $H$ be the Hamilton square coupling of $T_4$. By the construction of $H$, the edges in the two pairs $\{v_2^1v_3^2, v_3^1v_2^2\}$ and $\{v_2^1v_4^2, v_4^1v_2^2\}$ have opposite directions, so $H$ contains at least 2 edges with non-uniform orientation. Furthermore, the direction of the edge $v_4^1v_1^2$ may introduce additional minor direction  in $H$, which implies $M_4 = 3$. For $n=5$, we verify the result by computer, and the code is provided in Appendix \ref{appendix}.
\end{proof}

We are now in a position to prove our almost covering lemma.

\begin{lemma}[Almost covering lemma]\label{TH1}
    Suppose that  $0<1/n \ll 1/M \ll \varepsilon \ll d \ll\alpha\ll1.$  Let  $G$ be an  $n$-vertex oriented graph satisfying $\delta(G)\ge(2/3 + \alpha)n$. Then there exist a function $g(\delta(G),n)>n$ and a family of vertex-disjoint $2$-paths, denoted by $\mathcal{Q}$, covering all but at most $9\varepsilon n$ vertices. Furthermore, $\mathcal{Q}$ has the property that:
    \[
\sigma_{\max}(\mathcal{Q})\ge g(\delta(G),n)-O(d) n.
\]
\end{lemma}

\begin{proof}

Choose additional constants $ M'\in \mathbb{N}$ such that
     
   \[
       1/n \ll 1/M \ll 1/M' \ll \varepsilon \ll d\ll\alpha\ll 1.
    \]

    We first apply \cref{lem:diregularity} to \(G\) with parameters \((\varepsilon, d, M')\) to obtain a partition \(V_0, V_1, \ldots, V_k\) of \(V(G)\) with \(k \geq M'\). Let \(R\) be the reduced oriented graph of $G$ with parameters \((\varepsilon, d)\) given by \cref{dirprop},  the minimum degree of $R$ satisfies $\delta(R) \ge (2/3+\alpha/2)k$ by the choice of $\varepsilon \ll d\ll\alpha\ll 1$. Let $G^*$ be the pure oriented graph. 
    Clearly, there exists an integer $r$ such that
  \begin{equation}\label{deR}
      (1-\frac{1}{r-1})k\le\delta(R)\le (1-\frac{1}{r})k.
  \end{equation}

    \cref{HSnew1} yields a perfect tiling $\mathcal{T}$ composed of vertex-disjoint copies: $A_r$ copies of $T_r$ and $\overline{A}_r$ copies of $T_{r-1}$. A direct calculation  shows that  the total number of tiles satisfies 
\begin{equation}\label{T}
     |\mathcal{T}|=A_r + \overline{A}_r=k-\delta(R)\le k-(1-\frac{1}{r-1})k\le M\le dn. 
\end{equation}

  For \(r \ge 5\),  there exists a oriented Hamilton square in $T_r$, denoted by \(S_i\), together with some $T_3$ and $T_4$, we get a factor \(\mathcal{S} = \{S_1, S_2, \dots, S_{|\mathcal{T}|}\}\) of \(R\) (some of $S_i$ is isomorphic to $T_3$ or $T_4$). By \cref{unav}, we have 

\begin{equation}\label{factorinR2}
         \sigma_{\max}(\mathcal{S})\geq \sum_{i=1}^{A_r}\sigma_{\max}(S_i)+\sum_{j=1}^{\overline{A}_r}\sigma_{\max}(S_j)\ge A_r\cdot N_r+\overline{A}_r\cdot N_{r-1}.
    \end{equation}

Applying \cref{prop:super-regular} with $U=\mathcal{S}$ and $\Delta = 4$, we obtain an oriented subgraph $G_{\mathcal{S}}$ of $G^*$.
More precisely, $V(G_{\mathcal{S}})$ is the union of $V_{i}^{'}$, where each $V'_i \subset V_i$ is obtained by deleting exactly $8\varepsilon|V_i|$ vertices. Furthermore, for every edge $ij\in E(\mathcal{S})$, $G_{\mathcal{S}}[V_i^{'},V^{'}_j]$ is $(\sqrt{\varepsilon},d-16\varepsilon)$-super-regular.
Observe that all \( V_i' \) have the same order $(1 - 8\varepsilon)m$, where $m=|V_i|$, and that every balanced blow-up of a cycle is Hamiltonian. Therefore, an application of \cref{blowup} yields a 2-cycle factor $\mathcal{H}=\{H_1, H_2, \dots, H_{|\mathcal{T}|}\}$ of $G_{\mathcal{S}}$, where  each $H_i$ is the square of a Hamilton cycle and is \((1 - 8\varepsilon)m\)-blow-up of \( S_i \). Thus,
 \begin{equation}\label{sigH}
         \sigma_{\max}(\mathcal{H})\geq (1-8\varepsilon)m\cdot\sigma_{max}(\mathcal{S}).
    \end{equation}
    If we delete three consecutive edges from each 2-cycle in $\mathcal{H}$ , we obtain a family $\mathcal{Q}=\{Q_1,Q_2,\dots, Q_{|\mathcal{T}|}\}$ of 2-path that covers all but at most $|V_0|+\sum_{i=1}^{k}8\varepsilon|V_i|\leq 9\varepsilon n$ vertices in $V(G).$

We now estimate the number of dominant direction edges in   $\mathcal{Q}.$ By \cref{dirprop}, we have $\delta(R) \ge (\delta(G)/|G| - 2d)k$. Together with the fact that \( \sigma_{\max}(Q_i) \geq \sigma_{\max}(H_i) - 3 \) for each $i\in[|\mathcal{T}|]$, along with \eqref{T}, \eqref{factorinR2} and \eqref{sigH}, we obtain
    \begin{equation}\label{sigQ}
         \sigma_{\max}(\mathcal{Q})\ge (1-8\varepsilon)m\cdot(A_r\cdot N_r+\overline{A}_r\cdot N_{r-1})-3|\mathcal{T}|,
    \end{equation}
the right of (\ref{sigQ}) is greater than
    
   \begin{equation}\label{1}
       (1 - 9\varepsilon)n \left[ \Big((r-1)N_r - rN_{r-1}\Big) (\frac{\delta(G)}{n}-2d) - \Big((r-2)N_r - (r-1)N_{r-1}\Big) \right]-3|\mathcal{T}|.
   \end{equation}
Noting that,  $N_r\le N_{r-1}+r-1,$ and we have 
 \begin{equation}\label{2}
      (r-1)N_r - rN_{r-1}\le(r-1)^2.
 \end{equation}
Furthermore, since $r<k\le M$ ( $U(R)$ was covered by $K_r$ and $K_{r-1}$), and by the choice of parameter $1/n\ll 1/M\ll d$, combining with (\ref{1}) and (\ref{2}) one can get $$\sigma_{\max}(\mathcal{Q})\ge \left((r-1)N_r - rN_{r-1}\Big) \delta(G) - \Big((r-2)N_r - (r-1)N_{r-1}\right)n -O(d) n.$$

Writing \( g(\delta(G), n) \) in the form
\[
g(\delta(G),n)=
 \left((r-1)N_r - rN_{r-1}\Big) \delta(G) - \Big((r-2)N_r - (r-1)N_{r-1}\right)n.
\]
we can show that the lower bound of this lemma is nontrivial.



\begin{claim}\label{lem:fgeqn}
$g(\delta(G),n)-O(d)n > n$. 
\end{claim}
\begin{proof}\renewcommand*{\qedsymbol}{$\blacksquare$}
Recall that $A_r := (r-1) \delta(G) - (r-2)n$ and $\overline{A}_r := (r-1)n - r \delta(G)$, and note that $f(\delta(G),n) = A_r N_r + \overline{A}_r N_{r-1}$. 
For every $\delta(G)\geq2n/3$, there exists an integer $r$ such that
$$\left(1 - \frac{1}{r-1}\right)n \le \delta(G) \le \left(1 - \frac{1}{r}\right)n,$$
 we have $A_r \ge 0$ and $\overline{A}_r \ge 0$. Moreover,
 \begin{equation}\label{arar}
     A_r + \overline{A}_r 
= \bigl((r-1)\delta(G) - (r-2)n\bigr) + \bigl((r-1)n - r\delta(G)\bigr) 
= n - \delta(G) > 0,
 \end{equation}

and
$$rA_r + (r-1)\overline{A}_r = r\bigl((r-1)\delta(G) - (r-2)n\bigr) + (r-1)\bigl((r-1)n - r\delta(G)\bigr) = n.$$
Therefore,
$$
\begin{aligned}
g(\delta(G),n) - O(d)-n 
&= A_r N_r + \overline{A}_r N_{r-1} - \bigl(rA_r + (r-1)\overline{A}_r\bigr)-O(d)n \\
&= A_r(N_r - r) + \overline{A}_r\bigl(N_{r-1} - (r-1)\bigr)-O(d)n.\\
\end{aligned}
$$
Define $N_{\min} = \min\{N_r - r, N_r - (r - 1)\}$. By \cref{longpath} and \cref{Cr}, we have $N_{\min} > 0$. Together with \eqref{arar}, this implies that the right-hand side is bounded below by $N_{\min}\cdot(n - \delta(G)) - O(d)n$, which is strictly positive.
\end{proof}

By now, we have already proved that there exist a function $g(\delta(G),n)$ such that $\sigma_{\max}(\mathcal{Q})\ge g(\delta(G),n) -O(d) n.$
\end{proof}

\smallskip

\subsection{ Absorbing Lemma.}

The following lemma provides an absorbing 2-path for small set of leftover vertices.

\begin{lemma}[Absorbing lemma {[\citen{Levitt}]}]\label{TH2}
 \itshape
 Let $G$ be an $n$-vertex oriented graph with $\delta(G)\geq 2n/3,$ then there is a $2$-path $P_{abs}$ with at most $\alpha^9 n$ vertices such that,  for every $W\subseteq V(G)\backslash V(P_{abs})$ of size at most $\alpha^{20} n $, $G[P_{abs}\cup W]$ contains a spanning  $2$-path $P_{AW}$ having  the same end edges as $P_{abs}$.
\end{lemma}

\subsection{ Connecting and Reservoir Lemma.}

Next we proceed with the connecting lemma, which ensures the existence of short paths connecting prescribed vertex pairs, even when a small subset of vertices is forbidden.

\begin{lemma}[Connecting lemma {[\citen{Levitt}]}] \label{TH3}
\itshape 
 Let $G$ be an $n$-vertex oriented graph with $\delta(G)\geq 2n/3.$
For any two disjoint edges of $G$, $ab$ and $cd$, there exist a $2$-path with order $k\le 10/ \alpha^4$, which connects $ab$ and $cd$. Furthermore, this statement  remains true even if at most $\alpha^9 n$ vertices are  forbidden to be used on this connecting path.
\end{lemma}

After presenting  the absorbing and connecting lemmas, the next lemma allows us to efficiently connect many  end-edges with only a few vertices, without disturbing the existing structure of the graph.

\begin{lemma}[Reservoir lemma {[\citen{Levitt}]}] \label{TH4}
Let \(G\) be an \(n\)-vertex oriented  graph with $\delta(G)\ge 2n/3$. For any vertex subset \(W \subseteq V(G)\) with \(|W| \le \alpha^9 n\) there exists a set of vertices \(\mathcal{R} \subseteq V(G) \setminus W\) of size at most \(\alpha^{20}n/2\) having the following property:
For every \(S \subseteq \mathcal{R}\) with \(|S| \le \alpha^9 |\mathcal{R}|\) and for every two disjoint  edges $ab$ and $cd$ in $V(G)\backslash \mathcal{R} $, there is a $2$-path with order $k\le 10/ \alpha^4$ in \(\mathcal{R} \setminus S\).
\end{lemma}

\section{Proof of  Theorem~\ref{th1}}\label{proofmain}
Let \(\alpha\) be a constant with \(0 < 4\alpha \leq \xi\). If 
$\frac{2n}{3} \leq \delta(G) \leq \left( \frac{2}{3} + \frac{4\alpha}{3} \right)n,$
then \cref{posa} guarantees that \(G\) contains the square of a Hamilton cycle \(H\), for which \(\sigma_{\max}(H) \geq n\) holds trivially; taking \(f(\delta(G),n) = 3\delta(G)-n\) in \cref{th1}, the upper bound on \(\delta(G)\) yields
\[
\sigma_{\max}(H) \geq n \geq 3\left(\frac{2}{3} + \frac{4\alpha}{3}\right)n - n - \xi n \geq 3\delta(G) - n - \xi n,
\]
which implies the desired inequality. In the following, we consider $\delta(G)> (2/3+4\alpha/3) n$ and define

\[
f(\delta(G),n)=
\begin{cases}
3\delta(G)-n, \qquad \qquad\qquad\qquad \qquad\qquad\qquad \qquad\qquad
  \dfrac{2n}{3}\le \delta(G)\le \left(\dfrac{2}{3}+\dfrac{4\alpha}{3}\right)n,\\[1ex]
 \left((r-1)N_r - rN_{r-1}\Big) \delta(G) - \Big((r-2)N_r - (r-1)N_{r-1}\right)n
\quad \text{otherwise}.
\end{cases}
\]

\subsection{$\alpha$-extremal condition}

\begin{definition}[{[\citen{Levitt}]}] \label{a-ex}
  
    An oriented graph $G$ on $n$ vertices is \textit{$\alpha$-extremal} if there exist (not necessarily disjoint) $A,B\subseteq V(G)$ such that:
    
    \medskip
\begingroup
\setlength{\parindent}{0pt}
\setlength{\parskip}{0.45em} 
\textnormal{($i$)} $(\frac{1}{3}-\alpha)n\le|A|,|B|\le \frac{1}{3}n$,

\textnormal{($ii$)} $d_{U(G)}(A,B)<\alpha.$

\endgroup
    
\end{definition}

The following proposition uncovers the relation between $\delta(G)$ and the $\alpha$-extremal condition.

\begin{proposition}\label{clm:extremal}
Let $G$ be an oriented graph on $n$ vertices satisfies the $\alpha$-extremal condition, then $\delta\big(G\big)\le (2/3+4\alpha/3)n$.
\end{proposition}

\begin{proof}

By \cref{a-ex}, we have
$e(A,B) < \alpha |A||B|.$
Hence the average number of neighbors in $B$ of a vertex in $A$ is less than $\alpha|B|$:
\[
\frac{1}{|A|}\sum_{a\in A} d(a,B)
=\frac{e(A,B)}{|A|} < \alpha |B|.
\]
Therefore there exists $a\in A$ with
\begin{equation}\label{eq:small-cross-degree}
d(a,B) < \alpha |B|.
\end{equation}
Every neighbor of $a$ in $G$ lies either in $B$ or in $V(G)\setminus B$, so
\[
d(a)
= d(a,B) + d\big(a, V(G)\setminus B\big)
\le d(a,B) + (n-1-|B|).
\]
Combining with \eqref{eq:small-cross-degree} and $|B|\ge (1/3-\alpha)n$, we obtain
$\delta(G)\le (2/3+4\alpha/3)n.$
\end{proof}

\subsection{Completion of the proof of Theorem~\ref{th1}}\label{sec:proof2}

Suppose that
\[
1/n \ll 1/M \ll \varepsilon \ll d \ll 4\alpha \le \xi \ll 1.
\]

By \cref{clm:extremal}, $G$ is non-$\alpha$-extremal. 
Applying \cref{TH2} to get an absorbing 2-path \(P_{abs}\) with \(|V(P_{abs})| \le  \alpha^9 n\) such that, for every \(W \subseteq V(G) \setminus V(P_{abs})\) with \(|W| \le \alpha^{20} n\), \(G[V(P_{abs}) \cup W]\) contains a spanning 2-path having the same end edges as \(P_{abs}\).
Next we apply  \cref{TH4} to get a reservoir set   \(\mathcal{R} \subseteq V(G) \setminus V(P_{abs})\) with \(|\mathcal{R}| = \left\lfloor \frac{\alpha^{20}}{2} n \right\rfloor\).

    Let $G'=G\backslash(\mathcal{R}\cup V(P_{abs}))$ and $|G^\prime|=n^\prime$, clearly,
    $\delta(G')\geq (\frac{2}{3}+\frac{4\alpha}{3}-\alpha^9  - \frac{\alpha^{20}}{2} )n>(\frac{2}{3}+\alpha)n^\prime.$ 
    Applying \cref{TH1} to $G'$, we obtain a family $\mathcal{Q}$ of vertex-disjoint 2-paths, say $ \{Q_1, Q_2, \dots, Q_t\} $, where $t \le M$. These 2-paths cover $V(G')$ except for a vertex set of at most $9\varepsilon n $ vertices. Furthermore, \cref{TH1} guarantees that the $\mathcal{Q}$ satisfies: 
\begin{equation} \label{eq:cov_val}
     \sigma_{\max}(\mathcal{Q}) \ge f(\delta(G'),n')-O(d) n'.
\end{equation}

For convenience let us write \(Q_0 := P_{abs}\) and for every \(0 \le i \le t\), we say \(Q_i\) has initial edge $a_i$ and terminal edge $b_i$.
We then show how to find $H$ such that $\sigma_{\max}(H)\ge f(\delta(G),n)-\xi n$. To begin with, for the paths \(Q_0\) and \(Q_1\), by \cref{TH4} there exists an \((b_0, a_1)\)-path $L_0$ with at most $10/\alpha^4$ internal vertices from \(\mathcal{R}\). 

   We can greedily apply Lemma \ref{TH4} to find a \((b_i, a_{i+1})\)-path $L_i$ for every \(0 \le i \le t\) (write \(a_0 = a_{t+1}\)), since the total number of used vertices in the reservoir set \(\mathcal{R}\) is at most 
   \[
    \lvert \bigcup_{i=0}^{t} V(L_{i}) \rvert <t \cdot \frac{10}{\alpha^4}<\alpha^9 \lvert \mathcal{R}\rvert.
    \]
   
So far we have gotten the square of a cycle, say \(H_0\), that contains \(Q_0\), \(\mathcal{Q}\), and a fraction of vertices in \(\mathcal{R}\). 
Let \(W := V(G) \setminus V(H_0)\),
the total number of vertices outside $H_0$ is less than $9\varepsilon n+ \left\lfloor\frac{\alpha^{20}}{2} n\right\rfloor \le \alpha^{20} n.$
 Applying \cref{TH2} there exists a path \(P_{AW}\) spanning \(V(P_{abs}) \cup U\) having the same end-edges as \(P_{abs}\). 
Therefore, we get the square of a Hamilton cycle $H$ if we replace \(P_{abs}\) by \(P_{AW}\) in \(H_0\), as desired.

    Finally, we derive the number of edges in the dominant direction  of $H$. 
Throughout the entire proof, absorption and connection do not reduce the number of existing edges which belongs to $\sigma_{\max}(\mathcal{Q})$. Therefore, we directly neglect the impacts induced by them in the calculation, thereby arriving at following.
    \begin{align*}
    &\sigma_{\max}(H)
    \ge  \sigma_{\max}(\mathcal{Q}) 
    \overset{\mathclap{\eqref{eq:cov_val}}}{\ge}
    \Big((r-1)N_r - rN_{r-1}\Big) \delta(G^{\prime}) - \Big((r-2)N_r - (r-1)N_{r-1}\Big) n^\prime - O(d)n^\prime \\[2mm]
   &\ge \Big((r-1)N_r - rN_{r-1}\Big)(\delta(G)-(\alpha^9+\tfrac{\alpha^{20}}{2})n)
     - \Big((r-2)N_r - (r-1)N_{r-1}\Big)n- O(d)n\\[2mm]
    & \ge f(\delta(G),n)-\xi n.
   \end{align*}
   Furthermore, using the same method as in \cref{lem:fgeqn} , it can be shown that this is a nontrivial lower bound, i.e., $f(\delta(G),n)-\xi n > n$.

\subsection{Proof of Corollary~\ref{coro}}
Let $1/n \ll 1/M \ll \varepsilon \ll d \ll 4\alpha \le \xi \ll 1$
and \(\alpha\) is as defined in \cref{a-ex}. \cref{th1} has already established the case when \(2n/3 \leq \delta(G) \leq (2/3 + 4\alpha/3)n\); we then assume that \(\delta(G) \geq (2/3 + 4\alpha/3)n\).

Applying \cref{lem:diregularity} to \(G\) with parameters \((\varepsilon, d, M')\). By \cref{dirprop}, the minimum degree of its oriented reduced graph \(R\) satisfies \(\delta(R) \ge 2|R|/3\). Furthermore, \(U(R)\) is \(K_6\)-free. Indeed, if \(U(R)\) contains a \(K_6\) on clusters \(V_1,\ldots,V_6\), then each pair \((V_i,V_j)\) is \(\varepsilon\)-regular with density at least \(d\), so that  \cref{counting lemma} would yield a tournament on $6$ vertices in \(G\), contradicting the hypothesis. Hence \(U(R)\) is \(K_6\)-free. By Tur\'{a}n's theorem, \(U(R)\) has at most \((1-1/5)|R|^2/2\) edges, implying \(\delta(R)\le 4|R|/5\). Combining the lower and upper bounds on \(\delta(R)\) and following the argument of \cref{th1} with the parameter settings from \cref{Cr}, we obtain \(\sigma_{\max}(H)\ge 3\delta(G)-n-\xi n\), as desired.

\section{Concluding remarks}\label{sec:remark}

In this paper we have studied the oriented discrepancy problem for the square of a Hamilton cycle in oriented graphs with minimum degree at least \(2n/3\), thereby answering a question of Freschi and~Lo~\cite{Freschi}. Our proof combines the absorbing method with the directed regularity lemma.

 Several natural directions remain open. It would be interesting to determine the function that guarantees the existence of the square of a Hamilton cycle with $\sigma_{max}(H)\ge f(\delta(G),n)$. 

    \begin{problem}
        For every $\xi > 0$, there exist positive integers $r$ and $n_0$ such that the following holds. Determine  the largest function $f(\delta(G),n)$ such that, for every  oriented graph $G$ on $n$ vertices with $\delta(G) \ge 2n/3$, there exist the square of a Hamilton cycle $H$ satisfying $\sigma_{\max}(H)\ge f(\delta(G),n).$
    \end{problem}

    Recent work has extended Dirac-type discrepancy results to Ore-type conditions for Hamilton cycles~\cite{Chang1}. A natural question is whether, under an Ore-type condition (e.g., \(\sigma_2(G)\ge 4n/3\)), one can obtain a result analogous to \cref{th1}.

The function obtained in the main result of this paper heavily depends on the value of $M_n$ and $N_n$ in \cref{unav}, and its uncertainty is precisely why we cannot further advance using this approach. Therefore, another very interesting problem is to determine the exact value of \(M_n\) for $n\ge6$.

\newpage
\appendix

\section{Verification of $M_5=3$}\label{appendix}
In this section, we verify the equality $M_5=3$ using a Python program. The first part of the program enumerates all non-isomorphic tournaments on five vertices, and the second part checks that $M_5=3$.
\begin{lstlisting}[style=Python, title=``Generate 12 adjacency matrices of non-isomorphic tournaments $T_5$"]
import itertools

def edges_to_adj_matrix(n, edges_direction):
    adj = [[0] * n for _ in range(n)]
    for (i, j), dir in edges_direction.items():
        if dir == 1:
            adj[i][j] = 1
        else:
            adj[j][i] = 1
    return adj

def adj_to_canonical_fast(adj):
    n = len(adj)
    outdeg = [sum(row) for row in adj]
    vertices = list(range(n))
    vertices.sort(key=lambda x: outdeg[x])
    groups = []
    i = 0
    while i < n:
        j = i
        while j < n and outdeg[vertices[j]] == outdeg[vertices[i]]:
            j += 1
        groups.append(vertices[i:j])
        i = j
    best = None
    group_perms = [itertools.permutations(g) for g in groups]
    for perms in itertools.product(*group_perms):
        perm = []
        for p in perms:
            perm.extend(p)
        new_adj = [[0] * n for _ in range(n)]
        for i in range(n):
            for j in range(n):
                if adj[perm[i]][perm[j]]:
                    new_adj[i][j] = 1
        flat = tuple(new_adj[i][j] for i in range(n) for j in range(n))
        if best is None or flat > best:
            best = flat
    return best

def gen_tournaments_correct(n):
    m = n * (n - 1) // 2
    edges = [(i, j) for i in range(n) for j in range(i + 1, n)]
    seen = set()
    tournaments = []
    for bits in range(1 << m):
        dirs = {}
        for idx, (i, j) in enumerate(edges):
            dirs[(i, j)] = 1 if (bits >> idx) & 1 else 0
        adj = edges_to_adj_matrix(n, dirs)
        canon = adj_to_canonical_fast(adj)
        if canon not in seen:
            seen.add(canon)
            tournaments.append(adj)
    return tournaments

n = 5
tournaments = gen_tournaments_correct(n)
print(f"Number of non-isomorphic tournaments for n=5: {len(tournaments)}")

filename = 'T5'
with open(filename, 'w') as f:
    for adj in tournaments:
        adj_str = ';'.join(','.join(map(str, row)) for row in adj)
        f.write(f"{adj_str}\n")

print(f"Results saved to '{filename}'")

\end{lstlisting}

\begin{lstlisting}[style=Python, title=``Verification of $M_5$"]


import itertools

def read_tournaments_from_file(filename):
    tournaments = []
    with open(filename, 'r', encoding='utf-8') as f:
        lines = [line.strip() for line in f if line.strip()]

    for idx, line in enumerate(lines, 1):
        rows = line.split(';')
        n = len(rows)
        edges = []
        for i in range(n):
            cols = rows[i].split(',')
            for j in range(n):
                if int(cols[j]) == 1:
                    edges.append((i, j))
        tournaments.append((f"Tournament {idx}", edges))

    return tournaments

def build_adjacency_matrix(edges, n=5):
    adj = [[0] * n for _ in range(n)]
    for i, j in edges:
        adj[i][j] = 1
    return adj

def count_reversed_arcs_for_cycle(cycle, adj):
    n = len(cycle)
    consistent = 0
    
    for k in range(n):
        i = cycle[k]
        j = cycle[(k + 1) % n]
        if adj[i][j] == 1:
            consistent += 1
    
    for k in range(n):
        i = cycle[k]
        j = cycle[(k + 2) % n]
        if adj[i][j] == 1:
            consistent += 1
            
    total_edges = n * 2
    inconsistent = total_edges - consistent
    return min(consistent, inconsistent)

def find_min_reversed_arcs_for_tournament(edges, n=5):
    adj = build_adjacency_matrix(edges, n)
    vertices = list(range(n))
    all_cycles = []
    
    for perm in itertools.permutations(vertices[1:], n - 1):
        cycle = (0,) + perm
        all_cycles.append(cycle)

    min_rev = float('inf')
    for target_rev in range(0, n + 1):
        found = False
        for cycle in all_cycles:
            rev_count = count_reversed_arcs_for_cycle(cycle, adj)
            if rev_count == target_rev:
                min_rev = target_rev
                found = True
                break
        if found:
            break

    if min_rev == float('inf'):
        for cycle in all_cycles:
            rev_count = count_reversed_arcs_for_cycle(cycle, adj)
            if rev_count < min_rev:
                min_rev = rev_count

    return min_rev

def main():
    print("=" * 60)
    print("Computation of minimum reversed arcs for 12 tournaments on 5 vertices")
    print("=" * 60)

    filename = "T5"
    tournaments = read_tournaments_from_file(filename)
    print(f"Successfully read {len(tournaments)} tournaments\n")

    results = []
    for idx, (name, edges) in enumerate(tournaments, 1):
        print(f"Computing {name}...")
        min_rev = find_min_reversed_arcs_for_tournament(edges)
        results.append((name, min_rev))
        print(f"  Minimum reversed arcs = {min_rev}")

    rev_values = [rev for _, rev in results]
    max_val = max(rev_values)
    print(f"\n M5= {max_val}")

    output_file = "tournament_results.txt"
    with open(output_file, "w", encoding="utf-8") as f:
        for name, min_rev in results:
            f.write(f"{name}: {min_rev}\n")
        f.write("\n")
        f.write(f"m={max_val}")

    print(f"\nResults saved to '{output_file}'")
    print("\n" + "=" * 60)
    print("Computation finished!")

if __name__ == "__main__":
    main()
\end{lstlisting}

\end{document}